\documentclass[a4paper, 11pt]{article}

\usepackage{fullpage}

\usepackage{amsmath}
\usepackage{amsthm}
\usepackage{amsfonts}
\usepackage{tikz-cd}
\usepackage[utf8x]{inputenc}
\usepackage{enumerate}
\usepackage[english]{babel}
\usepackage{amssymb}
\usepackage{esint}
\usepackage{hyperref}
\usepackage{bm}
\usepackage{bbm}

\theoremstyle{plain}
\newtheorem{theorem}{Theorem}[section]
\newtheorem{lemma}[theorem]{Lemma}

\newtheorem{proposition}[theorem]{Proposition}
\theoremstyle{definition}
\newtheorem{definition}[theorem]{Definition}
\theoremstyle{remark}

\newtheorem{remark}[theorem]{Remark}

\DeclareMathOperator{\divv}{div}

\title{Semiflow selection \\for the compressible Navier--Stokes system}
\author{Danica Basari\'{c}}
\date{}

\begin{document}
	
	\maketitle
	
	\begin{center}
		Technische Universit\"{a}t Berlin \\
		Institute f\"{u}r Mathematik, Stra{\ss}e des 17. Juni 136, 10623 Berlin, Germany 
	\end{center}
	
	\begin{center}
		E-mail address: basaric@math.tu-berlin.de
	\end{center}

	\begin{abstract}
		Although the existence of dissipative weak solutions for the compressible Navier--Stokes system has already been established for any finite energy initial data, uniqueness is still an open problem. The idea is then to select a solution satisfying the \textit{semigroup property}, an important feature of systems with uniqueness. More precisely, we are going to prove the existence of a \textit{semiflow selection} in terms of the three state
		variables: the density, the momentum and the energy. Finally, we will show that it is
		possible to introduce a new selection defined only in terms of the initial density and
		momentum; however, the price to pay is that the semigroup property will hold almost
		everywhere in time.
	\end{abstract}
	
	\section{Introduction}
	
	Consider the compressible Navier--Stokes system
	\begin{equation} \label{continuity equation}
		\partial_t \varrho + \divv_x (\varrho \textbf{u}) =0,
	\end{equation}
	\begin{equation} \label{balance of momentum}
		\partial_t(\varrho \textbf{u}) + \divv_x(\varrho \textbf{u}\otimes \textbf{u})+ \nabla_x p(\varrho) = \divv_x \mathbb{S}(\nabla_x \textbf{u}),
	\end{equation}
	where $\varrho=\varrho(t,x)$ denotes the density, $\textbf{u}=\textbf{u}(t,x)$ the velocity, $p=p(\varrho)$ the pressure and $\mathbb{S}=\mathbb{S}(\nabla_x \textbf{u})$ the viscous stress. We will consider the system on the set $(t,x)\in (0,\infty)\times \Omega$, where $\Omega\subset \mathbb{R}^N$, $N=2,3$ is a bounded domain with $\partial \Omega$ of class $C^{2+\nu}$ for a certain $\nu>0$. As our goal is to handle a potentially ill--posed problem, we have deliberately omitted the case $N=1$, for which the problem is known to be be well posed, see Kazhikhov \cite{Kha}.
	
	We impose the no--slip boundary condition for the velocity
	\begin{equation}
		\textbf{u}|_{\partial \Omega}=0 \quad \mbox{for all } t\in [0,\infty),
	\end{equation}
	and we prescribe the initial conditions
	\begin{equation}
		\varrho(0,\cdot)=\varrho_0, \quad (\varrho\textbf{u})(0,\cdot)=(\varrho\textbf{u})_0.
	\end{equation}
	Finally, we assume a barotropic pressure $p\in C[0,\infty) \cap C^1(0,\infty)$ such that $p(0)=0$ and
	\begin{equation} \label{pressure}
		\begin{cases}
			p'(\varrho) \geq a_1 \varrho^{\gamma-1} -b &\mbox{for all } \varrho>0 \\
			p(\varrho) \leq a_2 \varrho^{\gamma} +b &\mbox{for all } \varrho\geq 0
		\end{cases}
	\end{equation}
	for certain constants $a_1>0$, $a_2$ and $b$, with $\gamma > \frac{N}{2}$ the adiabatic exponent, and the viscous stress tensor to be a linear function of the velocity gradient, more specifically to satisfy the Newton's rheological law
	\begin{equation} \label{viscous stress tensor}
		\mathbb{S}(\nabla_x \textbf{u})= \mu \left(\nabla_x \textbf{u}+\nabla_x^T \textbf{u}- \frac{2}{N}(\divv_x \textbf{u})\mathbb{I}\right) + \lambda(\divv_x \textbf{u})\mathbb{I},
	\end{equation}
	with $\mu >0$ and $\lambda\geq 0$. We would like to point out that \eqref{pressure} allows the pressure to a general 
	non--monotone function of the density. Still, as we shall see below, the problem admits global--in--time weak solutions and retains other fundamental properties of the system, notable the weak--strong uniqueness, see \cite{Cha}.
	
	We will consider dissipative weak solutions, i.e. solutions satisfying equations \eqref{continuity equation} and \eqref{balance of momentum} in a distributional sense along with the energy inequality, see Section \ref{dissipative solutions} below. Although the existence of global in time solutions has already been established for any finite energy initial data, see e.g. \cite{Lio} and \cite{Fei}, uniqueness is still an open task. Then, a natural question is whether it is possible or not to select a solution satisfying at least the \textit{semiflow} property, an important feature of systems with uniqueness:  letting the system run from time $0$ to time $s$ and then restarting and letting it run from time $s$ to time $t$ gives the same outcome as letting it run directly from time $0$ to time $t$.
	
	The result presented in this manuscript can be seen as the deterministic version of the stochastic paper done by Breit, Feireisl and Hofmanov\'{a} \cite{BreFeiHofm}. The construction of the semigroup  arises from the theory of Markov selection in order to study the well--posedness of certain systems; it was first developed by Krylov \cite{Kry} and later adapted by Flandoli and Romito \cite{FlaRom}, Cardona and Kapitanski \cite{CarKap} in the context of the incompressible Navier--Stokes system.
	
	Breit, Feireisl and Hofmanov\'{a} \cite{BreFeiHof} used the deterministic version motivated by \cite{CarKap} to show the existence of the semiflow selection for dissipative measure--valued solutions of the isentropic Euler system. Following the same strategy, we will establish the existence of a \textit{semiflow selection} for the compressible Navier--Stokes system \eqref{continuity equation}--\eqref{viscous stress tensor}. Specifically, introducing the momentum $\textbf{m}=\varrho \textbf{u}$, we show the existence of a measurable mapping
	\begin{equation*}
		V: [t, \varrho_0, \textbf{m}_0] \mapsto [\varrho(t), \textbf{m}(t)], \quad t\geq 0,
	\end{equation*}
	satisfying the semigroup property:
	\begin{equation} \label{semigroup property for almost every time}
		V[t_1+t_2, \varrho_0, \textbf{m}_0] = V \left[t_2, V[t_1, \varrho_0, \textbf{m}_0] \right] \quad \mbox{for a.e. } t_1,t_2 \geq 0,
	\end{equation}
	where $[\varrho,\textbf{m}= \varrho \textbf{u}]$ represents a dissipative weak solution to \eqref{continuity equation}--\eqref{viscous stress tensor}. At this stage, we would like to point out the main essential difference between the present paper and \cite{BreFeiHof}. The semigroup constructed for the Euler system in \cite{BreFeiHof} contains the total energy as one of the state variables. This may be seen as a kind of drawback as the energy should be determined in terms of the basic state variables $[\varrho, \textbf{m}]$. This is however a delicate issue for the Euler flow as the energy contains also the defect due to possible concentrations and/or oscillations. Such a problem does not occur for the Navier–Stokes system, where the energy is indeed a function of $[\varrho, \textbf{u}]$ at least for a.a. $t \in [0, \infty)$, cf. \eqref{semigroup property for almost every time}. 
	
	The paper is organized as follows. The remaining part of this section contains the definitions of a dissipative weak solution and admissibility. In Section \ref{set-up} we fix the topologies on the space of the initial data and the trajectory space, and we introduce the concept of a semiflow selection in terms of the three state variables: the density $\varrho_0$, the momentum $\textbf{m}_0$, and the energy $E_0$. In Section \ref{properties of U} we analyze the properties (compactness, non--emptiness, the \textit{shift invariance} and \textit{continuation} properties) of the solution set for a given initial data  while Section \ref{Semiflow selection} is devoted to the proof of the existence of a semiflow selection. Finally, in Section \ref{restriction} we study a new selection defined only in terms of the initial density $\varrho_0$ and the momentum $\textbf{m}_0$.

	\subsection{Dissipative weak solution} \label{dissipative solutions}
	Following \cite{Fei}, we can give the definition of a dissipative solution to the compressible Navier--Stokes system.
	
	\begin{definition} \label{dissipative solution}
		The pair of functions $\varrho$, $\textbf{u}$ is called \textit{dissipative weak solution} of the Navier--Stokes system \eqref{continuity equation}--\eqref{viscous stress tensor} with the total energy $E$ and initial data
		\begin{equation*}
			[\varrho_0, (\varrho \textbf{u})_0, E_0] \in L^{\gamma}(\Omega) \times L^{\frac{2\gamma}{\gamma+1}}(\Omega; \mathbb{R}^N) \times [0,\infty)
		\end{equation*}
		if the following holds:
		\begin{itemize}
			\item[(i)] \textit{regularity class}: 
			\begin{equation*}
				[\varrho, \varrho \textbf{u}, E] \in C_{weak,loc}([0,\infty); L^{\gamma}(\Omega)) \times C_{weak,loc}([0,\infty); L^{\frac{2\gamma}{\gamma+1}}(\Omega; \mathbb{R}^N))\times BV_{loc}([0,\infty)),
			\end{equation*}
			with $\varrho \geq 0$;
			\item[(ii)] \textit{weak formulation of the renormalized continuity equation}: for any $\tau>0$ and any functions
			\begin{equation*}
				B\in C[0,\infty) \cap C^1(0,\infty), \ b\in C[0,\infty) \mbox{ bounded on } [0,\infty),
			\end{equation*}
			\begin{equation*}
				B(0)=b(0)=0 \quad \mbox{and} \quad b(z)=zB'(z)-B(z) \mbox{ for any }z>0,
			\end{equation*}
			the integral identity
			\begin{equation} \label{weak formulation renormalized continuity equation}
				\left[ \int_{\Omega} B(\varrho) \varphi(t,\cdot) dx\right]_{t=0}^{t=\tau} = \int_{0}^{\tau} \int_{\Omega} [B(\varrho) \partial_t\varphi+ B(\varrho)\textbf{u}\cdot \nabla_x\varphi + b(\varrho)\divv_x \textbf{u} \varphi] dxdt,
			\end{equation}
			holds for any $\varphi \in C_c^1([0,\infty)\times \Omega)$, where $\varrho(0,\cdot)=\varrho_0$;
			\item[(iii)] \textit{weak formulation of the balance of momentum}: for any $\tau>0$ the integral identity
			\begin{equation} \label{weak formulation of the balance of momentum}
				\left[ \int_{\Omega} \varrho \textbf{u} \cdot \bm{\varphi}(t,\cdot) dx\right]_{t=0}^{t=\tau} = \int_{0}^{\tau} \int_{\Omega} [\varrho \textbf{u}\cdot \partial_t\bm{\varphi}+ (\varrho \textbf{u} \otimes \textbf{u}): \nabla_x \bm{\varphi}+ p(\varrho)\divv_x \bm{\varphi}- \mathbb{S}(\nabla_x \textbf{u}): \nabla_x \bm{\varphi}] dxdt,
			\end{equation}
			holds for any $\bm{\varphi} \in C_c^1([0,\infty)\times \Omega; \mathbb{R}^N)$, where $(\varrho \textbf{u})(0,\cdot)= (\varrho \textbf{u})_0$;
			\item[(iv)] \textit{energy inequality}: for a.e. $\tau\geq0$ we have 
			\begin{equation} \label{energy}
				E(\tau) = \int_{\Omega} \left[\frac{1}{2} \varrho |\textbf{u}|^2 + P(\varrho) \right](\tau, \cdot) dx,
			\end{equation}
			where the \textit{pressure potential} $P$ is chosen as a solution of 
			\begin{equation*}
				\varrho P'(\varrho)-P(\varrho)=p(\varrho);
			\end{equation*}
			we also require $E=E(\tau)$ to be a non-increasing function of $\tau$:
			\begin{equation} \label{energy inequality}
				[E\psi]_{t=\tau_1-}^{t=\tau_2+}-\int_{\tau_1}^{\tau_2} E(t)\psi'(t) dt + \int_{\tau_1}^{\tau_2} \psi \int_{\Omega} \mathbb{S}(\nabla_x \textbf{u}): \nabla_x\textbf{u} dxdt\leq 0,
			\end{equation}
			for any $0\leq \tau_1 \leq \tau_2$, $\psi \in C_c^1[0,\infty), \ \psi\geq0$, where $E(0-)=E_0$.
		\end{itemize}
	\end{definition}

	\begin{remark}
		Condition (ii) can be considered as a simple rescaling of the state variables in the continuity equation \eqref{continuity equation}; it is necessary in order to prove the weak sequential stability and the existence of dissipative weak solutions. In particular, choosing $B(z)=z$ we get the standard weak formulation of the continuity equation.
	\end{remark}
	\begin{remark}
		At this stage, similarly to \cite{BreFeiHof}, the total energy is considered as an additional phase variable - a non--increasing function of time possessing one sided limits at any time. In contrast with \cite{BreFeiHof}, the energy can be determined in terms of $\varrho$ and $\textbf{u}$, see \eqref{energy}, with the exception of a zero measure set of times.
	\end{remark}
	\begin{remark}
		The condition $E(0-) = E_0$ comes naturally from the assumption that the total energy is bounded at the initial time $t=0$, specifically
		\begin{equation*}
			E(0+) \leq E_0.
		\end{equation*}
		
	\end{remark}

	\subsection{Admissible solution}
	From now on, it is more convenient to work with the momentum $\textbf{m}=\varrho \textbf{u}$. Following \cite{BreFeiHof}, for a fixed initial data, we focus on a subclass of dissipative weak solutions consisting of the ones which \textit{minimize} the total energy. At the present state, we retain the total energy $E$ as an integral part of the solution so we work with the triples $[\varrho, \textbf{m},E]$. Finally, in Section \ref{restriction} we pass to the natural state variables $[\varrho, \textbf{m}]$. We introduce the relation
	\begin{equation*}
		[\varrho^1, \textbf{m}^1, E^1] \prec [\varrho^2, \textbf{m}^2, E^2] \ \Leftrightarrow \ E^1(\tau \pm) \leq E^2(\tau \pm) \mbox{ for any } \tau \in (0,\infty).
	\end{equation*}
	where $[\varrho^i, \textbf{m}^i, E^i]$, $i=1,2$ are two dissipative weak solutions sharing the same initial data $[\varrho_0, \textbf{m}_0, E_0]$.
	\begin{definition} \label{admissible solution}
		A dissipative weak solution $[\varrho, \textbf{m}, E]$ starting from the initial data $[\varrho_0, \textbf{m}_0, E_0]$ is said \textit{admissible} if it is minimal with respect to the relation $\prec$. More precisely, if  $[\tilde{\varrho}, \tilde{\textbf{m}}, \tilde{E}]$ is another dissipative solution starting from $[\varrho_0, \textbf{m}_0, E_0]$ and
		\begin{equation*}
			[\tilde{\varrho}, \tilde{\textbf{m}}, \tilde{E}] \prec [\varrho, \textbf{m}, E],
		\end{equation*}
		then
		\begin{equation*}
			E=\tilde{E}\ \mbox{ in }[0,\infty).
		\end{equation*}
	\end{definition}

	In particular, such selection criterion guarantees that equilibrium states belong to the class of dissipative weak solutions (see \cite{BreFeiHof}, Section 6.3).

	\section{Set--up} \label{set-up}
	
	First of all, we must choose suitable topologies on the space of the initial data and the space of dissipative weak solutions. For simplicity, we will consider the Hilbert space
	\begin{equation*}
		X= W^{-\ell,2}(\Omega) \times W^{-\ell,2}(\Omega; \mathbb{R}^N) \times \mathbb{R},
	\end{equation*}
	where the constant $\ell>\frac{N}{2}+1$ is fixed, along with its subset containing the initial data
	\begin{equation*}
		D= \left\{ [\varrho_0,\textbf{m}_0, E_0]\in X: \  \varrho_0 \in L^1(\Omega), \ \varrho_0\geq 0, \ \textbf{m}_0 \in L^1(\Omega; \mathbb{R}^N), \ \int_{\Omega}\left[\frac{1}{2}\frac{|\textbf{m}_0|^2}{\varrho_0}+ P(\varrho_0)\right]dx \leq E_0 \right\}.
	\end{equation*}
	Notice that the convex function $[\varrho,\textbf{m}] \mapsto \frac{|\textbf{m}|^2}{\varrho}$ is defined for $\varrho \geq 0$, $\textbf{m}\in \mathbb{R}^N$ as
	\begin{equation*}
		\frac{|\textbf{m}|^2}{\varrho}= \begin{cases}
			0 &\mbox{if }\textbf{m}=0, \\
			\frac{|\textbf{m}|^2}{\varrho} &\mbox{if }\varrho>0, \\
			\infty &\mbox{otherwise}.
		\end{cases}
	\end{equation*}
	If $[\varrho_0, \textbf{m}_0, E_0]\in D$, applying H\"{o}lder inequality, we can also deduce that $\varrho_0 \in L^{\gamma}(\Omega)$ and $\textbf{m}_0\in L^{\frac{2\gamma}{\gamma+1}}(\Omega; \mathbb{R}^N)$; accordingly, the set of the data can be seen as a closed convex subset of the Banach space $L^{\gamma}(\Omega) \times L^{\frac{2\gamma}{\gamma+1}}(\Omega; \mathbb{R}^N) \times \mathbb{R}$. Indeed, we can write
	\begin{equation*}
		D= \left\{ [\varrho_0,\textbf{m}_0, E_0]\in L^1_{+}(\Omega) \times L^1(\Omega; \mathbb{R}^N) \times \mathbb{R}: \  f([\varrho_0, \textbf{m}_0]) \leq E_0 \right\},
	\end{equation*}
	so that it coincides with the epigraph of the function $f:L^1_{+}(\Omega) \times L^1(\Omega; \mathbb{R}^N) \rightarrow [0,+\infty]$ such that
	\begin{equation*}
		f([\varrho_0, \textbf{m}_0]) = \int_{\Omega}\left[\frac{1}{2}\frac{|\textbf{m}_0|^2}{\varrho_0}+ P(\varrho_0)\right]dx.
	\end{equation*}
	Since $f$ is lower semi-continuous and convex, we obtain that its epigraph is closed and convex.
	
	As trajectory space, we will consider the separable space
	\begin{equation*}
		Q= C_{loc}([0,\infty); W^{-\ell,2}(\Omega)) \times C_{loc}([0,\infty); W^{-\ell,2}(\Omega; \mathbb{R}^N)) \times L^1_{loc}[0,\infty).
	\end{equation*}
	The choice of such topologies is justified by the fact that we want any dissipative weak solution $[\varrho, \textbf{m}, E]$, as defined in Definition \ref{dissipative solution}, to belong to the class $Q$ (since $\ell>\frac{N}{2}$, the $L^p$--space with $p\geq 1$ is compactly embedded in $W^{-\ell,2}$ so in particular it holds for $p=\gamma$ and $p=\frac{2\gamma}{\gamma+1}$ while equations \eqref{weak formulation renormalized continuity equation} and \eqref{weak formulation of the balance of momentum} give an information on the time regularity of the density and the momentum) but also to the set $D$ (this easily follows from the energy inequality) when evaluated at any time $t\geq 0$ in order to have the possibility to restart the system at a random time $t$. 
	
	Finally, for a fixed initial data $[\varrho_0, \textbf{m}_0, E_0]\in D$, we introduce the solution set
	\begin{align*}
		\mathcal{U}&[\varrho_0, \textbf{m}_0, E_0] \\
		&= \left\{ [\varrho, \textbf{m}, E] \in Q \ \big| \ [\varrho,\textbf{m}=\varrho \textbf{u}, E] \mbox{ is a dissipative  weak solution with initial data } [\varrho_0, \textbf{m}_0, E_0]\right\}.
	\end{align*} 
	
	\subsection{Semiflow selection -- main result}
	We can now define a semiflow selection to \eqref{continuity equation}--\eqref{viscous stress tensor}.
	
	\begin{definition} \label{semiflow selection}
		A \textit{semiflow selection} in the class of dissipative weak solutions for the compressible Navier--Stokes system \eqref{continuity equation}--\eqref{viscous stress tensor} is a mapping
		\begin{equation*}
			U: D \rightarrow Q, \quad U\{ \varrho_0, \textbf{m}_0, E_0 \} \in \mathcal{U}[\varrho_0, \textbf{m}_0, E_0] \mbox{ for any } [\varrho_0, \textbf{m}_0, E_0]\in D
		\end{equation*}
		enjoying the following properties:
		\begin{itemize}
			\item[(i)] \textit{Measurability.} The mapping $U: D \rightarrow Q$ is Borel measurable.
			\item[(ii)] \textit{Semigroup property.} We have
			\begin{equation*}
				U\{ \varrho_0,\textbf{m}_0, E_0 \}(t_1+t_2) = U \{ \varrho(t_1), \textbf{m}(t_1), E(t_1-) \}(t_2),
			\end{equation*}
			where $[\varrho, \textbf{m}, E]= U\{ \varrho_0, \textbf{m}_0, E_0 \}$ for any $[\varrho_0, \textbf{m}_0, E_0]\in D$ and any $t_1, t_2 \geq 0$.
		\end{itemize}
	\end{definition}
	
	We are now ready to state our main result; the proof is postponed to Section \ref{selection sequence}.
	\begin{theorem} \label{main result}
		The compressible Navier--Stokes system \eqref{continuity equation}--\eqref{viscous stress tensor} admits a semiflow selection $U$ in the class of dissipative weak solutions in the sense of Definition \ref{semiflow selection}. Moreover, we have that $U\{ \varrho_0, \textbf{m}_0, E_0 \}$ is admissible in the sense of Definition \ref{admissible solution}, for any $[\varrho_0,\textbf{m}_0, E_0]\in D$.
	\end{theorem}

	Theorem \ref{main result} is stated in terms of the three state variables $[\varrho, \textbf{m},E]$. In Section 
	\ref{restriction} below, we state a version of this result in terms of the natural state variables $[\varrho, \textbf{m}]$, see 
	Theorem \ref{final}. The price to pay is validity of the semigroup property for any time with the exception of a zero measure set.

	\section{Properties of $\mathcal{U}$} \label{properties of U}
	The set--valued map 
	\begin{equation*}
		D \ni [\varrho_0, \textbf{m}_0, E_0] \mapsto \mathcal{U}[\varrho_0, \textbf{m}_0,E_0] \in 2^Q,
	\end{equation*}
	introduced in the previous section, enjoys the following properties; in particular, the last two are the main tools we will need in order to construct the semiflow.
	\begin{itemize}
		\item[(\textbf{P1})] \textit{Non--emptiness}. For any $[\varrho_0, \textbf{m}_0, E_0] \in D$,
		\begin{equation*}
			\mathcal{U}[\varrho_0, \textbf{m}_0, E_0] \subset Q \mbox{ is non--empty}.
		\end{equation*} 
		
		This statement is equivalent in proving the \textit{existence} of a dissipative weak solution for any initial data $[\varrho_0, \textbf{m}_0, E_0]\in D$; more precisely, we have the following result.
		\begin{proposition} \label{existence}
			Let $[\varrho_0, \textbf{m}_0, E_0] \in D$ be given; then the Navier--Stokes system \eqref{continuity equation}-- \eqref{viscous stress tensor} admits a dissipative weak solution in the sense of Definition \ref{dissipative solution} with the initial data $[\varrho_0, \textbf{m}_0, E_0]$. 
		\end{proposition}
	
		For the proof see \cite{Fei}, Theorem 7.1.
		
		\item[(\textbf{P2})] \textit{Compactness.} For any $[\varrho_0, \textbf{m}_0, E_0] \in D$,
		\begin{equation*}
			\mathcal{U}[\varrho_0, \textbf{m}_0, E_0] \subset Q \mbox{ is compact}.
		\end{equation*} 
		
		This statement is equivalent in showing the \textit{weak sequential stability} of the solution set; specifically, the following result holds.
		
		\begin{proposition} \label{sequential stability}
			Suppose that $\{ \varrho_{0,\varepsilon}, \textbf{m}_{0,\varepsilon}, E_{0,\varepsilon} \}_{\varepsilon >0}\subset D$ is a sequence of data giving rise to a family of dissipative weak solutions $\{ \varrho_{\varepsilon}, \textbf{m}_{\varepsilon}, E_{\varepsilon} \}_{\varepsilon >0}$, that is, $[\varrho_{\varepsilon}, \textbf{m}_{\varepsilon}, E_{\varepsilon}]\in \mathcal{U}[\varrho_{0,\varepsilon},\textbf{m}_{0,\varepsilon}, E_{0,\varepsilon}]$. Moreover, we assume that the initial densities converge strongly 
			\begin{equation*}
			\varrho_{0,\varepsilon} \rightarrow \varrho_0 \quad \mbox{in } L^{\gamma}(\Omega)
			\end{equation*}
			and there exists  a constant $\overline{E}>0$ such that $E_{0,\varepsilon}\leq \overline{E}$ for all $\varepsilon>0$.
			
			Then, at least for suitable subsequences,
			\begin{equation*}
			\textbf{m}_{0,\varepsilon} \rightharpoonup \textbf{m}_0 \quad \mbox{in } L^{\frac{2\gamma}{\gamma+1}}(\Omega; \mathbb{R}^N), \quad E_{0,\varepsilon} \rightarrow E_0,
			\end{equation*}
			and
			\begin{equation*} 
			\begin{aligned}
			\varrho_{\varepsilon} \rightarrow \varrho \quad &\mbox{in } C_{weak,loc}([0,\infty); L^{\gamma}(\Omega)) \\
			\textbf{m}_{\varepsilon} \rightarrow \textbf{m} \quad &\mbox{in } C_{weak,loc}([0,\infty); L^{\frac{2\gamma}{\gamma+1}}(\Omega; \mathbb{R}^N))\\
			E_{\varepsilon}(\tau) \rightarrow E(\tau) \quad &\mbox{for every } \tau \in [0,\infty) \mbox{ and in } L^1_{loc}(0,\infty),
			\end{aligned}
			\end{equation*}
			where
			\begin{equation*}
			[\varrho, \textbf{m}, E] \in \mathcal{U}[\varrho_0, \textbf{m}_0, E_0].
			\end{equation*}
		\end{proposition}
		For the proof see \cite{Fei}, Theorems 6.1 and 6.2.
		
		\item[(\textbf{P3})] \textit{Measurability}. The mapping
		\begin{equation*}
			D \ni [\varrho_0, \textbf{m}_0, E_0] \mapsto \mathcal{U}[\varrho_0, \textbf{m}_0,E_0] \in 2^Q
		\end{equation*}
		is Borel measurable.
		
		Notice that, since $\mathcal{U}[\varrho_0, \textbf{m}_0, E_0]$ is a compact subset of the separable space $Q$ for any initial data $[\varrho_0, \textbf{m}_0, E_0] \in D$, requiring the Borel measurability of $\mathcal{U}$ is equivalent in proving the measurability with respect to the Hausdorff metric on the space of all compact subsets of $Q$. Due to Proposition \ref{sequential stability}, it is sufficient to apply the following lemma with $Y=D$ and $X=Q$.
		
		\begin{lemma}
			Let $Y$ be a metric space and $\mathcal{B}$  its Borel $\sigma$--field. Let $y \mapsto K_y$ be a map of $Y$ into $\mbox{Comp}(X)$ for some separable metric space $X$, with $\mbox{Comp}(X)$ the set of all the compact subsets of $X$. Suppose for any sequence $y_n \mapsto y$ and $x_n \in K_{y_n}$ , it is true that $x_n$ has a limit point $x$ in $K_y$. Then the map $y \mapsto K_y$ is a Borel map of $Y$ into $\mbox{Comp}(X)$.
		\end{lemma}
		The proof can be found in \cite{StrVar}, Lemma 12.1.8.
		
		\item[(\textbf{P4})] \textit{Shift invariance}. Introducing the positive shift operator for every $q\in Q$ as
		\begin{equation*}
			S_T \circ q, \ S_T\circ q(t)=q(T+t), \ t\geq 0, 
		\end{equation*}
		then, for any $[\varrho, \textbf{m}, E]\in \mathcal{U}[\varrho_0, \textbf{m}_0,E_0]$, we have
		\begin{equation*}
			S_T \circ [\varrho, \textbf{m}, E] \in \mathcal{U}[\varrho(T), \textbf{m}(T), E(T-)],
		\end{equation*}
		for any $T>0$.
		
		Instead of $E(T-)$, we could choose any $\mathcal{E}\geq E(T+)$ (recall the energy is non-increasing and thus in particular $E(T-)\geq E(T+)$); indeed, this more general result holds.
		
		\begin{lemma}
			Let $[\varrho_0, \textbf{m}_0, E_0]\in D$ and $[\varrho, \textbf{m},E]\in \mathcal{U}[\varrho_0, \textbf{m}_0, E_0]$. Then we have
			\begin{equation*}
			S_T \circ [\varrho, \textbf{m}, E] \in \mathcal{U}[\varrho(T), \textbf{m}(T), \mathcal{E}]
			\end{equation*}
			for any $T>0$, and any $\mathcal{E} \geq E(T+)$.
		\end{lemma}
		\begin{proof}
			A dissipative weak solution on the time interval $(0,\infty)$ solves also the same problem on $(T,\infty)$ with the initial data $[\varrho(T), \textbf{m}(T), E(T+)]$. Shifting the test functions in the integrals, this implies
			\begin{equation*}
				S_T \circ [\varrho, \textbf{m}, E] \in \mathcal{U}[\varrho(T), \textbf{m}(T), E(T+)].
			\end{equation*}
			Since the energy is non-increasing, we can choose every $\mathcal{E}\geq E(T+)$ as initial energy; indeed, everything will be well-defined
			\begin{equation*}
				S_T\circ E(0-) = \mathcal{E} \geq E(T+).
			\end{equation*}
		\end{proof}
	
		\item[(\textbf{P5})] \textit{Continuation}. Introducing the continuation operator for any $q_1, q_2 \in Q$ as 
		\begin{equation*}
			q_1 \cup_T q_2(t)= \begin{cases}
				q_1(t) & \mbox{for }0\leq t \leq T, \\
				q_2(t-T) &\mbox{for } t>T,
			\end{cases}
		\end{equation*}
		then, if $T>0$, and
		\begin{equation*}
			[\varrho^1, \textbf{m}^1, E^1] \in \mathcal{U}[\varrho_0, \textbf{m}_0, E_0], \ [\varrho^2, \textbf{m}^2, E^2] \in \mathcal{U}[\varrho^1(T), \textbf{m}^1(T), E^1(T-)],
		\end{equation*}
		then
		\begin{equation*}
			[\varrho^1, \textbf{m}^1, E^1] \cup_T [\varrho^2, \textbf{m}^2, E^2] \in \mathcal{U}[\varrho_0, \textbf{m}_0, E_0].
		\end{equation*}
		
		In this case, instead of $E^1(T-)$, we could choose any $\mathcal{E} \leq E^1(T-)$; indeed, this more general result holds.
		
		\begin{lemma}
			Let $[\varrho_0, \textbf{m}_0, E_0]\in D$ and 
			\begin{equation*}
			[\varrho^1, \textbf{m}^1, E^1] \in \mathcal{U}[\varrho_0, \textbf{m}_0, E_0], \ [\varrho^2, \textbf{m}^2, E^2] \in \mathcal{U}[\varrho^1(T), \textbf{m}^1(T), \mathcal{E}],
			\end{equation*}
			for some $\mathcal{E} \leq E^1(T-)$. Then
			\begin{equation*}
			[\varrho^1, \textbf{m}^1, E^1] \cup_T [\varrho^2, \textbf{m}^2, E^2] \in \mathcal{U}[\varrho_0, \textbf{m}_0, E_0].
			\end{equation*}
		\end{lemma}
		\begin{proof}
			We are simply pasting two solutions together at the time $T$, letting the second start from the point reached by the first one at the time $T$; thus the integral identities remain satisfied. Choosing the initial energy for $[\varrho^2, \textbf{m}^2, E^2]$ less or equal $E^1(T-)$, the energy of the solution $[\varrho^1, \textbf{m}^1, E^1] \cup_T [\varrho^2, \textbf{m}^2, E^2]$ remains non-increasing on $(0,\infty)$.
		\end{proof}
		
	\end{itemize}
	
	\section{Semiflow selection} \label{Semiflow selection}
	Starting from the family $\mathcal{U}[\varrho_0, \textbf{m}, E_0]$ of dissipative weak solutions for a fixed initial data $[\varrho_0, \textbf{m}_0, E_0] \in D$, the idea for the construction of the selection is to make this set smaller and smaller choosing the minima of particular functionals. More precisely, following the same arguments presented in \cite{BreFeiHof}, we consider the family of functionals
	\begin{equation*}
		I_{\lambda,F}[\varrho, \textbf{m},E]= \int_{0}^{\infty} e^{-\lambda t}F(\varrho(t),\textbf{m}(t),E(t)) dt, \quad \lambda>0,
	\end{equation*}
	where $F: X= W^{-\ell,2}(\Omega)\times W^{-\ell,2}(\Omega;\mathbb{R}^N) \times \mathbb{R}\rightarrow \mathbb{R}$ is a bounded and continuous functional. This choice is justified by the fact that $I_{\lambda,F}$ can be seen as Laplace transform of the functional $F$, an useful interpretation in the proof of the existence of the semiflow, as we will see in the next section.
	
	Given $I_{\lambda,F}$ and a set-valued mapping $\mathcal{U}$, we define a selection mapping $I_{\lambda,F}\circ \mathcal{U}$ by
	\begin{align*}
		&I_{\lambda,F}\circ \mathcal{U}[\varrho_0,\textbf{m}_0, E_0] \\
		&\quad = \{ [\varrho, \textbf{m},E] \in \mathcal{U}[\varrho_0,\textbf{m}_0, E_0] \ | \ I_{\lambda,F}[\varrho, \textbf{m},E] \leq I_{\lambda,F}[\tilde{\varrho}, \tilde{ \textbf{m}},\tilde{E}] \ \mbox{for all }  [\tilde{\varrho}, \tilde{ \textbf{m}},\tilde{E}] \in \mathcal{U}[\varrho_0,\textbf{m}_0, E_0] \}.
	\end{align*}
	Notice that a minimum exists since $I_{\lambda,F}$ is continuous on $Q$ and the set $\mathcal{U}[\varrho_0,(\varrho \textbf{u})_0, E_0]$ is compact in $Q$. We obtain the following result for the set $I_{\lambda,F}\circ \mathcal{U}$.
	\begin{proposition} \label{properties}
		Let $\lambda>0$ and $F$ be a bounded continuous functional on $X$. Let
		\begin{equation*}
			\mathcal{U}: [\varrho_0, \textbf{m}_0, E_0]  \in D \mapsto \mathcal{U}[\varrho_0, \textbf{m}_0, E_0] \subset 2^Q
		\end{equation*}
		be a multi-valued mapping having the properties (\textbf{P1}) -- (\textbf{P5}). Then the map $I_{\lambda,F} \circ \mathcal{U}$ enjoys  (\textbf{P1}) -- (\textbf{P5}) as well.
	\end{proposition}
	
	\begin{proof}
		\begin{itemize}
			
			\item[(\textbf{P1})] As already pointed out, $I_{\lambda, F} \circ \mathcal{U}[\varrho_0, \textbf{m}_0, E_0]\neq \emptyset$ since $I_{\lambda, F}$ is continuous on Q and the set $\mathcal{U}[\varrho_0, \textbf{m}_0, E_0]$ is a non--empty compact subset of $Q$.
			
			\item[(\textbf{P2})] Since $I_{\lambda,F}: \mathcal{U}[\varrho_0, \textbf{m}_0, E_0] \rightarrow \mathbb{R}$ is continuous and since the set of minima of a continuous function is closed (it is the counterpart of a point), we get that $I_{\lambda, F} \circ \mathcal{U}[\varrho_0, \textbf{m}_0, E_0]\subseteq \mathcal{U}[\varrho_0, \textbf{m}_0, E_0]$ is closed in a compact set and hence compact itself.
			
			\item[(\textbf{P3})] Notice that, since $I_{\lambda,F} \circ \mathcal{U}[\varrho_0, \textbf{m}_0, E_0]$ is a compact subset of the separable metric space $Q$ for any $[\varrho_0, \textbf{m}_0, E_0]\in D$, the Borel measurability of the multivalued mapping 
			\begin{equation*}
				[\varrho_0, \textbf{m}_0, E_0]  \in D \mapsto I_{\lambda, F} \circ \mathcal{U}[\varrho_0, \textbf{m}_0, E_0] \in \mathcal{K} \subset 2^Q
			\end{equation*}
			corresponds to measurability with respect to the Hausdorff metric on the space of all compact subsets of Q.
			
			In other words, let $d_H$ be the Hausdorff metric on the subspace $\mathcal{K} \subset 2^Q$ of all the compact subsets of $Q$:
			\begin{equation*}
				d_H(K_1,K_2) = \inf_{\varepsilon \geq 0} \{ K_1 \subset V_{\varepsilon}(K_2) \mbox{ and } K_2 \subset V_{\varepsilon}(K_1) \} \quad \mbox{for all }K_1,K_2 \in \mathcal{K},
			\end{equation*}
			where $V_{\varepsilon}(A)$ is the $\varepsilon$-neighborhood of the set $A$ in the topology of $Q$; then, it is enough to show that the mapping defined for all $K \in \mathcal{K}$ as
			\begin{equation*}
				\mathcal{I}_{\lambda, F}[K] = \{ z\in K | I_{\lambda, F}(z) \leq I_{\lambda, F}(\tilde{z}) \mbox{ for all }\tilde{z}\in K \} = \left\{ z\in K | \min_{z\in K} I_{\lambda, F}(z) \right\}
			\end{equation*}
			is continuous as a mapping on $\mathcal{K}$ endowed with the Hausdorff metric $d_H$. In particular we want to show that if $K_n \overset{d_H}{\longrightarrow} K$ with $K_n,K \in \mathcal{K}$ then $\mathcal{I}_{\lambda, F}[K_n] \overset{d_H}{\longrightarrow} \mathcal{I}_{\lambda, F}[K]$ for $n \rightarrow \infty$. More precisely, it is enough to show that for every $\varepsilon >0$ there exists $n_0=n_0(\varepsilon)$ such that
			\begin{equation} \label{double inclusion}
			\mathcal{I}_{\lambda, F}[K_n] \subset V_{\varepsilon}(\mathcal{I}_{\lambda, F}[K]) \quad \mbox{and} \quad \mathcal{I}_{\lambda, F}[K] \subset V_{\varepsilon}(\mathcal{I}_{\lambda, F}[K_n])
			\end{equation}
			for all $n\geq n_0$. First of all, notice that by the continuity of $I_{\lambda, F}$ we have
			\begin{equation} \label{contradiction}
			\min_{K_n} I_{\lambda, F} \rightarrow \min_{K} I_{\lambda, F} \quad \mbox{for } n\rightarrow \infty.
			\end{equation}
			We start proving the first inclusion of \eqref{double inclusion}. By contradiction, suppose that exists a sequence $\{z_n\}_{n\in \mathbb{N}}$ such that
			\begin{equation*}
			z_n \in K_n, \quad I_{\lambda, F}(z_n) = \min_{K_n} I_{\lambda, F}, \quad z_n \rightarrow z\in K \setminus V_{\varepsilon} (\mathcal{I}_{\lambda, F}[K]); 
			\end{equation*}
			in particular, $I_{\lambda, F}(z) > \min_{K} I_{\lambda, F}$. By the continuity of $I_{\lambda, F}$ we have
			\begin{equation*}
			\min_{K_n} I_{\lambda, F} = I_{\lambda,F} (z_n) \rightarrow I_{\lambda, F}(z) > \min_{K} I_{\lambda,F} \quad \mbox{for } n\rightarrow \infty; 
			\end{equation*}
			but this contradicts \eqref{contradiction}. Interchanging the roles of $K_n$ and $K$ we get the opposite inclusion in \eqref{double inclusion}. We get the claim.
			
			\item[(\textbf{P4})] We want to prove the shift invariance: for every $[\varrho_0, \textbf{m}_0, E_0] \in D$ and $[\varrho,\textbf{m}, E] \in I_{\lambda,F} \circ \mathcal{U}[\varrho_0, \textbf{m}_0, E_0]$,
			\begin{equation*}
				S_T \circ [\varrho,\textbf{m}, E] \in I_{\lambda, F} \circ\mathcal{U}[\varrho(T), \textbf{m}(T), E(T-)] \quad \mbox{for any } T>0.
			\end{equation*}
			Let $[\varrho^T, \textbf{m}^T, E^T] \in I_{\lambda, F} \circ\mathcal{U}[\varrho(T), \textbf{m}(T), E(T-)]$; then, since in particular
			\begin{align*}
				[\varrho, \textbf{m}, E] &\in  \mathcal{U}[\varrho_0, \textbf{m}_0, E_0], \\
				[\varrho^T, \textbf{m}^T, E^T] &\in \mathcal{U}[\varrho(T), \textbf{m}(T), E(T-)],
			\end{align*}
			and since $\mathcal{U}$ satisfies property (\textbf{A4}), we get
			\begin{equation*}
				[\varrho,\textbf{m}, E] \cup_T [\varrho^T,  \textbf{m}^T, E^T] \in \mathcal{U}[\varrho_0,  \textbf{m}_0, E_0].
			\end{equation*}
			From the choice of $[\varrho,\textbf{m}, E]$, which minimize $I_{\lambda, F}$ on $\mathcal{U}[\varrho_0,  \textbf{m}_0, E_0]$, we obtain
			\begin{equation} \label{condition}
				I_{\lambda, F} [\varrho,\textbf{m}, E] \leq I_{\lambda, F} ([\varrho,\textbf{m},E]\cup_T [\varrho^T,  \textbf{m}^T, E^T]).
			\end{equation}
			Hence, using \eqref{condition} in the fifth line and the definition of $\cup_T$ in the sixth line
			\begin{align*}
				I_{\lambda, F} (S_T \circ [\varrho, \textbf{m}, E]) &= \int_{0}^{\infty} e^{-\lambda t} F(S_T \circ [\varrho, \textbf{m}, E](t)) dt \\
				&= \int_{0}^{\infty} e^{-\lambda t} F([\varrho, \textbf{m}, E](t+T)) dt \\
				&= e^{\lambda T} \int_{T}^{\infty} e^{-\lambda s} F([\varrho, \textbf{m}, E](s)) ds \\
				&= e^{\lambda T} \left( I_{\lambda, F}[\varrho,  \textbf{m}, E] - \int_{0}^{T} e^{-\lambda s} F([\varrho, \textbf{m}, E](s)) ds\right) \\
				& \leq e^{\lambda T} \left(I_{\lambda, F} ([\varrho, \textbf{m},E]\cup_T [\varrho^T,  \textbf{m}^T, E^T]) - \int_{0}^{T}e^{-\lambda s} F([\varrho, \textbf{m}, E](s)) ds  \right) \\
				& = e^{\lambda T} \int_{T}^{\infty} e^{-\lambda s}F([\varrho^T, \textbf{m}^T, E^T](s-T)) ds \\
				& = e^{\lambda T} \int_{0}^{\infty} e^{-\lambda (t+T)} F([\varrho^T, \textbf{m}^T, E^T](t)) dt \\
				& = I_{\lambda, F} [\varrho^T,\textbf{m}^T, E^T].
			\end{align*}
			This implies that $S_T \circ [\varrho, \textbf{m}, E]$ minimizes $I_{\lambda,F}$ and consequently  belongs to \\ $I_{\lambda, F} \circ\mathcal{U}[\varrho(T), \textbf{m}(T), E(T-)]$ for any $T>0$. 
			
			\item[(\textbf{P5})] We want to prove the continuation: if $T>0$ and $[\varrho^1,\textbf{m}^1,E^1]\in I_{\lambda,F}\circ \mathcal{U}[\varrho_0,\textbf{m}_0, E_0]$, $[\varrho^2,\textbf{m}^2,E^2]\in I_{\lambda, F}\circ\mathcal{U}[\varrho^1(T),\textbf{m}^1(T), E^1(T-)]$, then
			\begin{equation*}
				[\varrho^1, \textbf{m}^1, E^1] \cup_T [\varrho^2,\textbf{m}^2,E^2] \in I_{\lambda, F} \circ \mathcal{U}[\varrho_0,\textbf{m}_0, E_0].
			\end{equation*}
			Using the shift invariance for $\mathcal{U}$ we obtain
			\begin{equation*}
				S_T \circ [\varrho^1,\textbf{m}^1, E^1] \in \mathcal{U}[\varrho^1(T),\textbf{m}^1(T), E^1(T-)];
			\end{equation*}
			since $[\varrho^2,\textbf{m}^2, E^2]$ is a minimum of $I_{\lambda,F}$ on $\mathcal{U}[\varrho^1(T),\textbf{m}^1(T), E^1(T-)]$ we get 
			\begin{equation} \label{condition2}
				I_{\lambda, F} [\varrho^2,\textbf{m}^2, E^2] \leq I_{\lambda,F} (S_T \circ [\varrho^1,\textbf{m}^1, E^1]).
			\end{equation}
			Hence, using \eqref{condition2} in the fourth line
			\begin{align*}
				I_{\lambda,F} &([\varrho^1, \textbf{m}^1, E^1] \cup_T [\varrho^2,\textbf{m}^2,E^2] ) \\
				&= \int_{0}^{T} e^{-\lambda t} F([\varrho^1,\textbf{m}^1, E^1](t)) dt + \int_{T}^{\infty} e^{-\lambda t} F([\varrho^2,\textbf{m}^2, E^2] (t-T)) dt \\
				&= \int_{0}^{T} e^{-\lambda t} F([\varrho^1,\textbf{m}^1, E^1](t)) dt + e^{-\lambda T} \int_{0}^{\infty} e^{-\lambda s} F([\varrho^2,\textbf{m}^2, E^2] (s)) ds \\
				&= \int_{0}^{T} e^{-\lambda t} F([\varrho^1,\textbf{m}^1, E^1](t)) dt + e^{-\lambda T} I_{\lambda,F} [\varrho^2,\textbf{m}^2, E^2] \\
				&\leq \int_{0}^{T} e^{-\lambda t} F([\varrho^1,\textbf{m}^1, E^1](t)) dt + e^{-\lambda T} I_{\lambda,F} (S_T \circ [\varrho^1,\textbf{m}^1, E^1]) \\
				&= \int_{0}^{T} e^{-\lambda t} F([\varrho^1,\textbf{m}^1, E^1](t)) dt + e^{-\lambda T} \int_{0}^{\infty} e^{-\lambda s} F ([\varrho^1,\textbf{m}^1, E^1](s+T)) ds \\
				&= \int_{0}^{T} e^{-\lambda t} F([\varrho^1,\textbf{m}^1, E^1](t)) dt + \int_{T}^{\infty} e^{-\lambda t} F([\varrho^1, \textbf{m}^1, E^1](t)) dt \\
				&= I_{\lambda,F} [\varrho^1,\textbf{m}^1, E^1].
			\end{align*}
			Using the continuation property for $\mathcal{U}$, we have that 
			\begin{equation*}
				[\varrho^1, \textbf{m}^1, E^1] \cup_T [\varrho^2,\textbf{m}^2,E^2]\in \mathcal{U}[\varrho_0,\textbf{m}_0,E_0],
			\end{equation*} 
			and since $[\varrho^1,\textbf{m}^1, E^1]$ is a minimum of $I_{\lambda, F}$, we must have 
			\begin{equation*}
				I_{\lambda,F} ([\varrho^1, \textbf{m}^1, E^1] \cup_T [\varrho^2,\textbf{m}^2,E^2] )= I_{\lambda,F} [\varrho^1,\textbf{m}^1, E^1];
			\end{equation*}
			thus $[\varrho^1, \textbf{m}^1, E^1] \cup_T [\varrho^2,\textbf{m}^2,E^2] \in I_{\lambda, F} \circ \mathcal{U}[\varrho_0, \textbf{m}_0, E_0]$.
		\end{itemize}
	\end{proof}
	
	\subsection{Selection sequence} \label{selection sequence}
	
	In this section we will prove the existence of the semiflow selection for the compressible Navier--Stokes system. We will need the following topological result, which is a variation of the Cantor's intersection theorem.
	
	\begin{theorem} \label{Cantor intersection theorem}
		Let $S$ be a Hausdorff space. A decreasing nested sequence of non-empty compact subsets of $S$ is non-empty. In other words, supposing $\{ C_k \}_{k\in \mathbb{N}}$ is a sequence of non-empty compact subsets of $S$ satisfying 
		\begin{equation*}
		C_0 \supseteq C_1 \supseteq C_2 \supseteq \dots \supseteq C_k \supseteq \dots
		\end{equation*}
		it follows that
		\begin{equation*}
		\bigcap_{k\in \mathbb{N}} C_k \neq \emptyset.
		\end{equation*}
	\end{theorem}
	\begin{proof}
		By contradiction, assume $\bigcap_{k\in \mathbb{N}} C_k = \emptyset$. For each $n$, let $U_n=C_0 \setminus C_n$; since
		\begin{equation*}
		\bigcup_{n\in \mathbb{N}} U_n = \bigcup_{n\in \mathbb{N}} (C_0\setminus C_n) = C_0 \setminus \left( \bigcap_{n\in \mathbb{N}} U_n\right)
		\end{equation*}
		and $\bigcap_{n\in \mathbb{N}} C_n = \emptyset$, we obtain $\bigcup_{n\in \mathbb{N}}U_n = C_0$. Since $C_0 \subset S$ is compact and $\{ U_n \}_{n\in \mathbb{N}}$ is an open cover (on $C_0$) of $C_0$, we can extract a finite cover $\{ U_{n_1}, \dots, U_{n_m} \}$. Let $U_k$ be the largest set of this cover ($C_k$ the correspondent smallest set), which exists by the ordering hypothesis on the collection $\{ C_n \}_{n\in \mathbb{N}}$. Then 
		\begin{equation*}
		C_0 \subset \bigcup_{j=1}^{m} U_{n_j} = \bigcup_{j=1}^{m} (C_0 \setminus C_{n_j}) = C_0 \setminus \bigcap_{j=1}^{m} C_{n_j}= C_0 \setminus C_k = U_k.
		\end{equation*}
		Then $C_k = C_0 \setminus U_k= \emptyset$, a contradiction since every set of the sequence $\{ C_n \}_{n\in \mathbb{N}}$ is non-empty by hypothesis.
	\end{proof}
	
	We are now ready to prove Theorem \ref{main result}.
	\begin{proof} [Proof of Theorem \ref*{main result}]
		First of all, we will select only those solutions that are admissible, meaning minimal with respect to the relation $\prec$ introduced in Definition \ref{admissible solution}. To this end, it is sufficient to consider the functional $I_{\lambda, \alpha}$ with $\alpha(\varrho,\textbf{m}, E)=\alpha(E)$,
		\begin{equation} \label{conditions on alpha}
			\alpha: \mathbb{R}\rightarrow \mathbb{R} \mbox{ smooth, bounded, and strictly increasing}.
		\end{equation}
		Indeed, if $[\varrho, \textbf{m}, E]\in I_{\lambda, \alpha} \circ \mathcal{U}[\varrho_0, \textbf{m}_0, E_0]$ then
		\begin{equation} \label{condition 3}
			\int_{0}^{\infty} e^{-t} \alpha(E(t)) dt \leq \int_{0}^{\infty} e^{-t} \alpha(\tilde{E}(t)) dt
		\end{equation}
		for any $[\tilde{\varrho}, \tilde{ \textbf{m}}, \tilde{E}]\in \mathcal{U}[\varrho_0, \textbf{m}_0, E_0]$. Now, proceeding by contradiction, suppose that $[\tilde{\varrho}, \tilde{\textbf{m}}, \tilde{E}]\in \mathcal{U}[\varrho_0,  \textbf{m}_0, E_0]$ is such that $[\tilde{\varrho}, \tilde{ \textbf{m}}, \tilde{E}] \prec [\varrho, \textbf{m}, E]$, that is, $\tilde{E} \leq E$ in $[0,\infty)$. Then, since $\alpha$ is strictly increasing, $\alpha(\tilde{E}(t))\leq \alpha(E(t))$ for every $t\in [0,\infty)$, which implies that $e^{-t}[\alpha(E(t))- \alpha(\tilde{E}(t))]\geq0$. Using the monotonicity of the integral, we obtain
		\begin{equation*}
		\int_{0}^{\infty} e^{-t} [\alpha(E(t))- \alpha(\tilde{E}(t))] dt \geq 0;
		\end{equation*}
		on the other side, condition \eqref{condition 3} tells us that
		\begin{equation*}
		\int_{0}^{\infty} e^{-t} [\alpha(E(t))- \alpha(\tilde{E}(t))] dt \leq 0.
		\end{equation*}
		The only possibility is to have the equality in both the integral relations above and thus $\alpha(E(t)) = \alpha(\tilde{E}(t))$ for a.e. $t\in (0,\infty)$; since $\alpha$ is strictly increasing, this implies $E=\tilde{E}$ a.e. in $(0,\infty)$. 
		
		Next, we choose a countable basis $\{ \textbf{e}_n \}_{n\in \mathbb{N}}$ in $L^2(\Omega; \mathbb{R}^N)$, and a countable set $\{ \lambda_k \}_{k\in \mathbb{N}}$ which is dense in $(0,\infty)$. We consider a countable family of functionals,
		\begin{align*}
			I_{k,0}[\varrho, \textbf{m}, E] &= \int_{0}^{\infty} e^{-\lambda_kt} \alpha(E(t)) dt, \\
			I_{k,n}[\varrho, \textbf{m}, E]  &= \int_{0}^{\infty} e^{-\lambda_kt} \alpha\left( \int_{\Omega} \textbf{m}(t,\cdot) \cdot \textbf{e}_n dx \right) dt,
		\end{align*}
		where again $\alpha$ satisfies condition \eqref{conditions on alpha}; the functionals are well defined since $ \textbf{m}(t,\cdot)\in W^{-\ell, 2}(\Omega; \mathbb{R}^N)$ for all $t$. Let $\{ (k(j),n(j)) \}_{j=1}^{\infty}$ be an enumeration of all the involved combinations of indices, that is, an enumeration of the countable set
		\begin{equation*}
		(\mathbb{N} \times \{ 0 \}) \cup (\mathbb{N} \times \mathbb{N}).
		\end{equation*}
		We define
		\begin{equation*}
			\mathcal{U}^j = I_{k(j),n(j)} \circ \dots \circ I_{k(1),n(1)} \circ I_{1,\alpha} \circ \mathcal{U}, \quad j=1,2, \dots,
		\end{equation*}
		and 
		\begin{equation*}
			\mathcal{U}^{\infty} = \bigcap_{j=1}^{\infty} \mathcal{U}^j.
		\end{equation*}
		The set-valued mapping
		\begin{equation*}
		[\varrho_0, \textbf{m}_0, E_0] \in D \mapsto \mathcal{U}^{\infty} [\varrho_0,  \textbf{m}_0, E_0] 
		\end{equation*}
		enjoys the properties (\textbf{P1})--(\textbf{P5}). Indeed:
		\begin{enumerate}
			
			\item[(\textbf{P1})--(\textbf{P2})] first, notice that for every fixed initial data $[\varrho_0, \textbf{m}_0, E_0]\in D$ the sets $\mathcal{U}^j[\varrho_0, \textbf{m}_0, E_0]$ are nested:
			\begin{equation*}
				I_{1,\alpha} \circ \mathcal{U}[\varrho_0, \textbf{m}_0, E_0] \supseteq \mathcal{U}^1[\varrho_0, \textbf{m}_0, E_0] \supseteq \dots \supseteq \mathcal{U}^j[\varrho_0, \textbf{m}_0, E_0] \supseteq \dots.
			\end{equation*}
			By Proposition \ref{properties} we can deduce that $I_{1, \alpha} \circ \mathcal{U}[\varrho_0, \textbf{m}_0, E_0]$ is compact, and iterating this procedure we obtain that all $\mathcal{U}^j[\varrho_0, \textbf{m}_0, E_0]$ are compact. Since $Q$ is a Hausdorff space, every compact set is also closed and a countable intersection of closed set is closed. Since $\mathcal{U}^{\infty}[\varrho_0, \textbf{m}_0, E_0] \subseteq I_{1, \alpha} \circ \mathcal{U}[\varrho_0, \textbf{m}_0, E_0]$, which is compact, we obtain that $\mathcal{U}^{\infty}[\varrho_0, \textbf{m}_0, E_0]$ is compact. By Proposition \ref{properties} we can also deduce that every $\mathcal{U}^j[\varrho_0, \textbf{m}_0, E_0]$ is non-empty; applying Theorem \ref{Cantor intersection theorem} we then get that $\mathcal{U}^{\infty}[\varrho_0, \textbf{m}_0, E_0] \neq \emptyset$;
			
			\item[(\textbf{P3})]as it is an intersection set--valued map obtained from measurable set--valued maps, it is also measurable;
			
			\item[(\textbf{P4})] to prove the shift property, let $[\varrho_0, \textbf{m}_0, E_0] \in D$ and $[\varrho, \textbf{m}, E]\in \mathcal{U}^{\infty}[\varrho_0, \textbf{m}_0, E_0]$; then, in particular $[\varrho, \textbf{m}, E]\in \mathcal{U}^j[\varrho_0, \textbf{m}_0, E_0]$ for every $j$. By Proposition \ref{properties}, we can deduce that $I_{1, \alpha} \circ \mathcal{U}$ satisfies the shift invariance property, and iterating this procedure we obtain that this holds for every $\mathcal{U}^j$. This implies
			\begin{equation*}
				S_T \circ [\varrho, \textbf{m}, E] \in \mathcal{U}^j[\varrho(T), \textbf{m}(T), E(T-)], \mbox{ for all }j \mbox{ and all }T>0.
			\end{equation*}
			Thus
			\begin{equation*}
				S_T \circ [\varrho, \textbf{m}, E] \in \mathcal{U}^{\infty}[\varrho(T), \textbf{m}(T), E(T-)], \mbox{ for all }T>0;
			\end{equation*}
			
			\item[(\textbf{P5})] to prove the continuation property, let $T>0$, $[\varrho^1,\textbf{m}^1,E^1]\in  \mathcal{U}^{\infty}[\varrho_0,\textbf{m}_0, E_0]$ and  $[\varrho^2,\textbf{m}^2,E^2]\in \mathcal{U}^{\infty}[\varrho^1(T), \textbf{m}^1(T), E^1(T-)]$; then, in particular we have $[\varrho^1,\textbf{m}^1,E^1]\in  \mathcal{U}^j[\varrho_0, \textbf{m}_0, E_0]$ and  $[\varrho^2, \textbf{m}^2,E^2]\in \mathcal{U}^j[\varrho^1(T),\textbf{m}^1(T), E^1(T-)]$ for every $j$. By  Proposition \ref{properties}, we can deduce that $I_{1, \alpha} \circ \mathcal{U}$ satisfies the continuation property, and iterating this procedure we obtain that this holds for every $\mathcal{U}^j$. This implies
			\begin{equation*}
			[\varrho^1, \textbf{m}^1, E^1] \cup_T [\varrho^2, \textbf{m}^2,E^2] \in \mathcal{U}^j[\varrho_0, \textbf{m}_0, E_0] \mbox{ for all }j \mbox{ and all }T>0.
			\end{equation*}
			Thus
			\begin{equation*}
			[\varrho^1, \textbf{m}^1, E^1] \cup_T [\varrho^2, \textbf{m}^2,E^2] \in \mathcal{U}^{\infty}[\varrho_0,\textbf{m}_0, E_0] \mbox{ for all }T>0.
			\end{equation*}
		\end{enumerate}
		We claim that for every $[\varrho_0,  \textbf{m}_0, E_0] \in D$ the set $\mathcal{U}^{\infty}$ is a singleton, meaning
		\begin{equation*}
		\mathcal{U}^{\infty}[\varrho_0, \textbf{m}_0, E_0]= U\{ \varrho_0, \textbf{m}_0, E_0 \} \in Q. 
		\end{equation*}
		To verify this, we observe that
		\begin{equation*}
			I_{k(j),n(j)} [\varrho^1,\textbf{m}^1, E^1] = I_{k(j),n(j)} [\varrho^2,\textbf{m}^2, E^2]
		\end{equation*}
		for any $[\varrho^1, \textbf{m}^1, E^1], [\varrho^2,  \textbf{m}^2, E^2]\in \mathcal{U}^{\infty}[\varrho_0, \textbf{m}_0, E_0]$ for all $j=1,2, \dots$; from the choice of $\{ k(j), n(j) \}_{j\in \mathbb{N}}$, we can see the integrals $I_{k(j), n(j)}$ as Laplace transforms 
		\begin{equation*}
			F(\lambda_k)= \int_{0}^{\infty} e^{-\lambda_k t}f(t)dt
		\end{equation*}
		of the functions
		\begin{equation*}
			f\in \left\{ \alpha(E), \  \alpha\left(\int_{\Omega}  \textbf{m}\cdot \textbf{e}_n dx \right)\right\}.
		\end{equation*}
		We can apply Lerch's theorem: if a function $F$ has the inverse Laplace transform $f$, then $f$ is uniquely determined (considering functions which differ from each other only on a point set having Lebesgue measure zero as the same). Then we get that
		\begin{align*}
		\alpha(E^1(t)) &= \alpha(E^2(t)), \\
		\alpha\left(\int_{\Omega} \textbf{m}^1(t,\cdot)\cdot \textbf{e}_n dx \right) &= \alpha\left(\int_{\Omega} \textbf{m}^2(t,\cdot)\cdot \textbf{e}_n dx \right),
		\end{align*}
		for all $n \in \mathbb{N}$ and for a.e. $t\in (0,\infty)$. As $\alpha$ is strictly increasing we must in particular have
		\begin{equation*}
			E^1(t)= E^2(t), \quad \langle \textbf{m}^1(t,\cdot); \textbf{e}_n \rangle_{L^2(\Omega; \mathbb{R}^N)}= \langle \textbf{m}^2(t,\cdot); \textbf{e}_n \rangle_{L^2(\Omega; \mathbb{R}^N)},
		\end{equation*}
		for all $n \in \mathbb{N}$ and for a.e. $t\in (0,\infty)$. Since $\{ \textbf{e}_n \}_{n\in \mathbb{N}}$ form a basis in $L^2(\Omega; \mathbb{R}^N)$ we conclude
		\begin{equation*}
			\textbf{m}^1=\textbf{m}^2, \mbox{ and } E^1=E^2 \mbox{ a.e. on }(0,\infty).
		\end{equation*}
		From the continuity equation \eqref{continuity equation} and from the fact that $\varrho^1(0,\cdot)= \varrho^2(0,\cdot)$ it is easy to see that
		\begin{equation*}
			\varrho^1= \varrho^2 \quad \mbox{a.e. on } (0,\infty).
		\end{equation*}
		It remains to prove that $U$ is a semiflow selection: measurability follows from (\textbf{P3}) while the semigroup property follows from (\textbf{P4}): for $t_1,t_2 \geq 0$ it holds 
		\begin{equation*}
		U\{ \varrho_0, \textbf{m}_0, E_0 \} (t_1+t_2)= S_{t_1} \circ U\{ \varrho_0, \textbf{m}_0, E_0 \} (t_2) = U\{ \varrho(t_1), \textbf{m}(t_1), E(t_1 -) \} (t_2).
		\end{equation*}
		This completes the proof.
	\end{proof}

	\section{Restriction to semigroup acting only on the initial data} \label{restriction}
	
	As a matter of fact, the semiflow selection $U=U\{\varrho_0,\textbf{m}_0, E_0\}$ is determined in terms of the \textit{three} state variables: the density $\varrho_0$, the momentum $\textbf{m}_0$, and the energy $E_0$. Introduction of the energy might be superfluous; indeed, as pointed out in \eqref{energy}
	\begin{equation*}
		E(\tau) = \int_{\Omega} \left[ \frac{1}{2}  \frac{|\textbf{m}|^2}{\varrho} +P(\varrho)\right] (\tau,\cdot) dx \quad \mbox{for a.e. } \tau\geq 0.
	\end{equation*}
	The point is that the equality holds with the exception of a zero measure set of times. More specifically, the energy $E(\tau)$ is a non-increasing function with well-defined right and left limits $E(\tau\pm)$, while
	\begin{equation*}
		\int_{\Omega} \left[ \frac{1}{2} \frac{|\textbf{m}|^2}{\varrho}+ P(\varrho)\right](\tau,\cdot) dx 
	\end{equation*}
	is defined at \textit{any} $\tau$ in terms of weakly continuous functions $t\mapsto \varrho(t,\cdot)$, $t\mapsto \textbf{m}(t,\cdot)$. Due to the convexity of the superposition 
	\begin{equation*}
		[\varrho, \textbf{m}] \mapsto \frac{1}{2} \frac{|\textbf{m}|^2}{\varrho} +P(\varrho)
	\end{equation*} 
	the function 
	\begin{equation*}
		\tau \mapsto \int_{\Omega} \left[ \frac{1}{2} \frac{|\textbf{m}|^2}{\varrho} +P(\varrho)\right] (\tau,\cdot) dx
	\end{equation*}
	is lower semi--continuous in $\tau$. In particular,
	\begin{equation*}
		\int_{\Omega} \left[ \frac{1}{2} \frac{|\textbf{m}|^2}{\varrho} +P(\varrho)\right] (\tau,\cdot) dx \leq E(\tau \pm) \quad \mbox{for any } \tau,
	\end{equation*}
	where equality holds with the exception of a set of time of measure zero. 
	
	We may introduce a new selection defined only in terms of the initial data $\varrho_0$, $\textbf{m}_0$; however, the price to pay is that the semigroup property will hold almost everywhere in time. More specifically, we can state this final result.
	
	\begin{theorem} \label{final}
		Let $U=U\{ \varrho_0, \textbf{m}_0, E_0 \}$ be the semiflow selection associated to the Navier--Stokes system in the sense of Definition \ref{semiflow selection}. Consider the set of initial data 
		\begin{equation*}
			\widetilde{D} = \left\{ [\varrho_0, \textbf{m}_0]: \  \left[ \varrho_0, \textbf{m}_0, \int_{\Omega} \left(\frac{1}{2} \frac{|\textbf{m}_0|^2}{\varrho_0} +P(\varrho_0) \right) dx \right] \in D\right\}.
		\end{equation*}
		Defining $V:\widetilde{D} \rightarrow Q$  such that 
		\begin{equation*}
				V\{ \varrho_0, \textbf{m}_0 \} (t) = U\left\{ \varrho_0, \textbf{m}_0,  \int_{\Omega} \left[ \frac{1}{2} \frac{|\textbf{m}_0|^2}{\varrho_0} +P(\varrho_0)\right] dx \right\} (t)
		\end{equation*}
		for all $t\in (0,\infty)$, then $V$ will satisfy the semigroup property only almost everywhere; more precisely, calling $\mathcal{T} \subset (0,\infty)$ the set of times defined as
		\begin{equation*}
			\mathcal{T} = \left\{ \tau \in (0,\infty):  \ E(\tau)= \int_{\Omega} \left[\frac{1}{2}\frac{|\textbf{m}|^2}{\varrho} +P(\varrho)\right](\tau,\cdot) dx \right\},
		\end{equation*}
		then $\mathcal{T}$ is a set of full measure and 
		\begin{equation*}
			V\{ \varrho_0, \textbf{m}_0 \} (t_1+t_2) = V\{ V\{\varrho_0, \textbf{m}_0\} (t_1) \} (t_2)
		\end{equation*}
		holds for all $t_1,t_2 \in \mathcal{T}$.
	\end{theorem}

	\bigskip
	
	\centerline{\bf Acknowledgement}
	
	This work was supported by the Einstein Foundation, Berlin. The author wishes to thank Prof. Eduard Feireisl for the helpful advice and discussions.

\end{document}